\documentclass{article}
\usepackage{amsfonts}

\newtheorem{theorem}{Theorem}[section]

\newtheorem{definition}[theorem]{Definition}
\newtheorem{example}[theorem]{Example}

\newtheorem{problem}[theorem]{Problem}
\newtheorem{proposition}[theorem]{Proposition}
\newtheorem{remark}[theorem]{Remark}

\newenvironment{proof}[1][Proof]{\noindent\textbf{#1.} }{\ \rule{0.5em}{0.5em}}

\def\RR{\mathbb{R}}

\def\EE{\mathbb{E}}

\def\cB{{\cal B}}

\def\cF{{\cal F}}

\def\de{{\delta}}
\def\la{{\lambda}}
\def\si{{\sigma}}

\def\Om{{\Omega}}
\def\al{{\alpha}}

\def\Ga{{\Gamma}}

\def\de{{\delta}}

\def\si{{\sigma}}

\def\tr{{ \hbox{ Tr} }}

\def\la{{\lambda}}

\def\vare{{\varepsilon}}
\def\vE{{\cal E}}
\def \eref#1{\hbox{(\ref{#1})}}

\def\Om{{\Omega}}
\def\om{{\omega}}
\def\cS{{\cal S}}
\def \comb#1#2{{\left({#1}\atop{#2}\right)}}
\def\hatotimes{{\hat {\otimes}}}

\begin{document}

\title{ Wick Calculus For Nonlinear Gaussian Functionals  }
\author{
Yaozhong Hu\thanks{%
Y. Hu was supported in part by the National Science Foundation under
Grant No.   DMS0504783, and the International Research Team on
Complex
Systems, Chines Academy of Sciences. }  \ \  and Jia-An Yan\thanks{%
J.-A. Yan was supported by the National Natural Science Foundation
of China (No. 10571167), the National Basic Research Program of
China (973 Program) (No.2007CB814902), and the Science Fund for
Creative Research Groups (No.10721101). }
\\
Department of Mathematics\thinspace ,\ University of Kansas\\
405 Snow Hall\thinspace ,\ Lawrence, Kansas 66045-2142\\
}
\date{}
\maketitle

\begin{abstract} This paper   surveys  some results on  Wick product and Wick
renormalization. The framework is the abstract Wiener space.  Some
known results on    Wick product and Wick renormalization in the
white noise analysis framework are presented for classical random
variables. Some conditions are described for random variables whose
Wick product or whose renormalization are integrable random
variables.   Relevant results on multiple Wiener integrals, second
quantization operator, Malliavin calculus and their relations with
the Wick product and Wick renormalization are also briefly
presented. A useful tool for Wick product is the $S$-transform which
is also described without the introduction of generalized random
variables.
\end{abstract}


{\bf 2000 MR Subject Classification}:  60G15, 60H05, 60H07, 60H40

\section{Introduction}
In the constructive Euclidean quantum field theory, such as the
$P(\phi)$ theory or the  $:\phi^4:$ theory, there have been
encountered {\it infinite}  quantities, which originated (from
mathematical point of view)   from the product of generalized
functions (see \cite{simon}, \cite{glimm}, or \cite{kallianpur} and
the references therein for more details). To obtain useful
information out of these infinite quantities, Wick (\cite{wick})
first introduced the now so-called {\it Wick renormalization}.
According to  \cite{holdenoksendal},   the {\it Wick product}  in
stochastic analysis  was first introduced  by Hida and Ikeda
\cite{hida}.   Meyer and Yan (\cite{meyeryan}) extended to cover the
Wick products of Hida distributions.  Now Wick product is applied to
stochastic differential equations (\cite{huoksendal}), stochastic
partial differential equations, stochastic quantization
(\cite{kallianpur}) and many other fields.

In stochastic analysis, most of the research work on Wick product
are on Hida distribution spaces or other spaces of generalized
random variables. In this paper we   survey   some results that we
have frequently used.  To make  the concept of Wick product  accessible to
broader  audience, we restrict ourselves to the classical
framework, namely, the classical random variables.  In fact, a motivation to  write such a  survey
is from some students who need  to know some results relevant to Wick product and the
way of how to use them  in their research.
After we have done some research on the references, we found out that
many results on Wick product have been already known by the second author
in the framework of Hida distribution
(\cite{yan91}, \cite{yan95}). But they are relatively
unknown to the experts on the field.

In Section 2, we introduce the framework and some
results which are useful in Wick product.  In particular, we introduce
the multiple integrals and the chaos expansion.

Malliavin calculus is very useful in the problems where Wick product
presents. In Section 3, we give a simplest presentation of some
results in Malliavin calculus which are relevant to wick product.

Wick product is introduced in Section 4. We present some  basic
results.  Some of them may  be new.

In Euclidean quantum field theory, the Wick renormalization is more
widely used. In section 5, we present some results  on Wick renomalization.
In fact, in Euclidean quantum field theory a very special
abstract Wiener space and a very special random variables are needed
(see \cite{kallianpur}).  However, we will not go into detail.

This paper is a condensed survey of some results on Wick product.
We do not intend to give a survey on the historical account.
So in some cases,  for a concept or a result, probably not the original
references are going to be cited.

\setcounter{equation}{0}

\section{ Multiple integrals and chaos expansion}

Let $H$  be a separable Hilbert space with scalar product
$\langle\cdot,\cdot \rangle _H$. There is a  Banach space $B$ ($B$ is not  unique)
with the following properties.
\begin{description}
\item{(i)}\quad $H  $ is continuously embedded in $B$ and   $H$ is dense in $B$.
The dual $B'$ (the space of continuous linear functionals) of $B$
is identified as a (dense) subspace of $H$ ($B'\subset B$).
\item{(ii)}\quad There is  a Borel measure
$\mu$ on $(B, \cB)$, where $\cB$ is the Borel $\si$-algebra of $B$
such that
\begin{equation}
\int_{B} \exp \{ i \langle l, x\rangle \} d\mu(x) =
\exp \{ -\|l\|_H^2 /2\},  \quad l\in \Phi'\,,
\end{equation}
where $\langle\cdot, \cdot\rangle $  means the pairing  between $B$ and $B'$
(namely, $\langle l , x\rangle=l(x)$).
\end{description}
 The triple $(B, H, \mu)$
is an {\it abstract Wiener space}.  We   denote
$\EE (f)=\int_{B} f(x) d\mu(x)$ and
$L^p=L^p(B, \cB, \mu)$.

For any $l\in B'$,  $\langle l, \cdot\rangle : B\to \RR$ is a mean zero
Gaussian random variable with variance $\|l\|_H^2$.  By a limiting argument,
for any $l\in H$, $\langle l, \cdot\rangle : B\to \RR$ can be defined as a
Gaussian random variable, denoted by $\tilde l$.

Fixed an  $n$ and  introduce the symmetric tensor product $H^{\hatotimes
n}$ by the following procedure.

Let $\{ e_1,   e_2, \cdots \}$  be an orthonormal basis  of $H$ and
let  $\hatotimes$ denote  the symmetric tensor product.  Then
\begin{equation}
f_n=\sum_{\hbox { finite}} f_{i_1, \cdots, i_n} e_{i_1}\hatotimes\cdots
\hatotimes e_{i_n}, \quad  f_{i_1, \cdots, i_n}\in \RR\label{e.2.2}
\end{equation}
is an element of $H^{\hatotimes n}$ with the Hilbert norm
\begin{equation}
\|f_n\|_{H^{\hatotimes n}}^2 =\sum_{\hbox { finite}} |f_{i_1, \cdots,
i_n}|^2\,,
\end{equation}
$H^{\hatotimes n}$  is the completion  of  all the elements of
above form under the above norm.

To define the multiple integral, we need to use the Hermite polynomials.
Let
\[
H_n(x)=(-1)^n e^{\frac{x^2}{2}}\frac{d^n}{dx^n} e^{-\frac{x^2}{2}}
=\sum_{k\le n/2} \frac{(-1)^k n!}{2^k k!(n-2k)!}
x^{n-2k}\,,\quad x\in \RR
\]
be the $n$-th Hermite polynomial ($n=0, 1, 2, \cdots$).  Its generating function is
\[
e^{tx-\frac{t^2}{2}}=\sum_{n=0}^\infty \frac{t^n}{n!} H_n(x)\,.
\]
Any element $f_n$  in  $H^{\hatotimes n}$ of the form
\eref{e.2.2} can be rewritten   as
\begin{equation}
 f_n  =\sum_{\hbox { finite}} f_{j_1, \cdots, j_m} e_{j_1}^{\hat\otimes k_1}\hatotimes\cdots
\hatotimes e_{j_m}^{\hat\otimes k_m}, \quad  f_{j_1, \cdots, j_m}\in
\RR,\,,\label{e.2.4}
\end{equation}
where $j_1\,, \cdots, j_m$ are different.
For this type of integrands the multiple integral is defined as
\begin{equation}
I_n(f_n):= \sum_{\hbox { finite}} f_{j_1, \cdots, j_m} H_{j_1}\left(\tilde e_{j_1}\right)
 \cdots H_{j_m}\left(\tilde e_{j_m}\right)\,.\label{e.2.5}
\end{equation}
In particular, we have for $f\in H$
\begin{equation}
I_n(f^{\otimes n})= \|f\|_H^nH_n(\|f\|_H^{-1}\tilde
f)\,.\label{e.2.6}
\end{equation}
 For general element $f_n$ in $H^{\hatotimes
n}$ we can define the multiple integral $I_n(f_n)$  by the $L^2$
convergence. It is straightforward to obtain the following isometry
equality
\begin{equation}
\EE |I_n(f_n)|^2 =n! \|f_n\|_{H^{\hatotimes n}}^2\,. \label{e.2.7}
\end{equation}
One can also construct the  Fock space $\Phi (H)$ on $H$ as follows.
\[
\Phi (H)=\bigoplus_{n=0}^\infty  H^{\hatotimes n}\,.
\]
The scalar product of two elements $f=(f_0, f_1, f_2, \cdots)$ and $g=(g_0, g_1, g_2, \cdots)$
in $\Phi(H)$ is defined as
\begin{equation}
\langle f\,, g\rangle_{\Phi(H)}  =\sum_{n=0}^\infty n! \langle
f_n\,, g_n\rangle _{H^{\hatotimes n}}\,. \label{e.2.8}
\end{equation}
The chaos expansion theorem states that any square integrable random variable $F$ on
$(B, \cB, \mu)$ can be written as
\begin{equation}
F=\sum_{n=0}^\infty I_n(f_n)\,, \quad f=(f_0, f_1, f_2, \cdots)\in
\Phi(H) \label{e.2.9}
\end{equation}
and
\begin{equation}
\EE(F^2)=\|f\|^2_{\Phi(H)} =\sum_{n=0}^\infty  n!
\|f_n||_{H^{\hatotimes n}}^2\,. \label{e.2.10}
\end{equation}
We refer to \cite{meyer}  and the references
therein for further  details.

\begin{example} Let $B=H=\RR^d$ and let  $\mu$ be the standard Gaussian measure
on $B$.
Then $(B, \cB, \mu)$ is the $d$-dimensional  standard Gaussian measure space.
\end{example}

\begin{example} Let
\[
H=\left\{ f:[0, T]\rightarrow \RR\,; \ f(0)=0\ \hbox{ $f$ is absolutely continuous on }\ [0, T]\right\}\,.
\]
It is a Hilbert space under the norm $\langle f\,, \  g\rangle =\int_0^T f'(t) g'(t) dt$.
Let
\[
\Om= \left\{ f:[0, T]\rightarrow \RR\,; \ f(0)=0\ \hbox{ $f$ is  continuous on }\ [0, T]\right\}
\]
with the sup norm.
Then $(\Om, \cF, \mu)$ is a canonical Wiener space, where $\cF$ is the Borel $\si$-algebra on $\Om$
(with respect to the sup norm).
\end{example}

\begin{example} Consider a domain  $D$  of $\RR^d$ ($d$  dimensional Euclidean
space).  Together with  some nice boundary conditions (if $D$ is not
the whole  space $\RR^d$)  we can prove (see \cite{glimm}) that
there is a kernel $K(x, y)$  such that the following equation
(with  mass $m=1$)  holds
$$
(-\Delta +1) K(x,y) =\delta (x-y)\,,
$$
where $\Delta$ is the Dirichlet Laplacian on $D$.
A Hilbert space of (generalized) functions can be determined by
by
\[
\langle f\,, g\rangle =\int_{D} K(x,y) f(x) f(y) dxdy\,,
\]
where $f$ and $g$ are two (generalized)  real-valued functions.
The Gaussian measure associated with this Hilbert space is useful in the
Euclidean quantum field theory  (see also \cite{kallianpur}
and the references therein).
\end{example}

\begin{definition} Let $f_n\in H^{\hatotimes n}$.  $f_n$ is called {\it negative  definite}  on
$H^{\hatotimes n} $  if
\[
\left\langle f_n,
h^{\hat\otimes n}\right\rangle_{H^{\hatotimes n}}\le 0 \qquad \forall \ \ h\in H
\]
\end{definition}

\begin{example} The following  are proved in \cite{kallianpur}.
\begin{description}
\item{(i)}\quad If    $f_2\in H^{\hatotimes 2}$,  then    there is an
$\al>0$ such that $\EE\exp [\al I_2(f_2)] <\infty$.
And   $\EE\exp [\al I_2(f_2)] <\infty$  is true for all $\al >0$ iff
$f_2$  is negative definite  (see \cite{kallianpur}, Theorem 5.1).
\item{(ii)}\quad For any nonzero $f_{2n+1} $ in $H^{\hatotimes 2n+1}$ ($n\ge 1$) and any $\la\in \RR$,
 $\EE\left(\exp\left\{ \la I_{2n+1}(f_{2n+1})\right\}\right)=\infty$.
 (see \cite{kallianpur},  Theorem 5.2).
 \end{description}
 It is   conjectured  (see \cite
 {kallianpur})  that
 \begin{description}
\item{(iii)}\quad  If $n$ is  even,  then
$\EE\exp [ I_n(f_n)] <\infty$  iff   $f_n$ is negative  definite on
$H^{\hatotimes n} $.
\end{description}
\end{example}

\begin{definition}
If $F$ has a chaos expansion $F=\sum_{n=0}^\infty F_n $ and $\al\in \RR$, then  the second
quantization operator of $\al$ acting on $F$ is defined as
\begin{equation}
\Ga(\al)F=\sum_{n=0}^\infty \al^n F_n\,.
\end{equation}
\end{definition}

For this operator we have the following famous theorem.
\begin{theorem} (Nelson's hypercontractivity) Let $1\le p<q<\infty$.  The following inequality holds
\begin{equation}
\|\Ga(\al)F\|_q \le \| F\|_p\,, \qquad \forall \ F\in L^p
\end{equation}
if and only if $|\al|\le \sqrt{\frac{p-1}{q-1}}$.
\end{theorem}
This inequality was first obtained by Nelson and appears in many places.
See \cite{meyer2}, \cite{unified} and the references therein for
further detail.

The following theorem is due to \"Ust\"unel and Zakai \cite{ustunel}  and
see \cite{kallenberg} for a  simpler  proof.
\begin{theorem} Let $f\in H^{\hatotimes n}$ and $g\in H^{\hatotimes m}$.  Then
$I_n(f)$ and $I_m(g)$ are independent if and only if
\begin{equation}
\langle f  \,, g \rangle _{H}=0\,,
\end{equation}
where $\langle f  \,, g \rangle _{H}\in H^{\hat\otimes {n+m-2}}$
defined by
\[
\langle f  \,, g \rangle _{H}=\sum_{n=1}^\infty \langle f  \,,
e_n\rangle _{H}\hat\otimes \langle g  \,, e_n\rangle _{H} \,.
\]
\end{theorem}

For an  $f\in H^{\hatotimes n}$ satisfying some more
conditions on the existence of trace of $f$,
the {\it multiple Stratonovich integral}  $S_n(f)$ can also be introduced
in the following way:

Let $\{e_1, e_2, \dots\}$ be an orthonormal basis of $H$.
Let $f\in H^{\hatotimes n}$ and consider the following random variable:
\begin{equation}
S_n^N(f )=\sum_{k_1\,, \cdots\,,  k_n=1}^N \langle f  \,,
e_{k_1}\hat\otimes \cdots\hat\otimes e_{k_n}\rangle_{H^{\hat\otimes
n}} \tilde e_{k_1} \cdots \tilde e_{k_n}\,. \label{e.2.sn}
\end{equation}

\begin{definition} If as $N\rightarrow \infty$, $S_n^N(f )$ converges in $L^2$,
then we say the multiple Stratonovich integral of $f $ exists. The limit is called
the multiple Stratonovich integral of $f $ and is  denoted by $S_n(f )$.
\end{definition}
\begin{definition} Denote
\[
\tr^{k, N}  f_n =\sum_{i_1, \cdots, i_k=1}^N  \langle f\,, e_{i_1}\hatotimes  e_{i_1} \hatotimes
\cdots \hatotimes e_{i_k}\hatotimes e_{i_k}\rangle_{H^{\hatotimes 2k}}
\]
which is considered as an element in $H^{\hatotimes (n-2k)}$.
If  as $N\rightarrow \infty$, $\tr^{k, N}  f_n $ converges in $H^{\hatotimes (n-2k)}$, then we say the  trace
of order $k$ exists and denote it by
\[
\tr^k f=\sum_{i_1, \cdots, i_k=1}^\infty \langle f\,, e_{i_1}\hatotimes  e_{i_1} \hatotimes
\cdots \hatotimes e_{i_k}\hatotimes e_{i_k}\rangle_{H^{\hatotimes 2k}}\,.
\]
\end{definition}

If the traces   of order $k$ of $f$ exist for all $k\le n/2$,
then the multiple Stratonovich integral of $f $, namely $S_n(f )$, exists and the
 following Hu-Meyer formula holds
\begin{eqnarray}
S_n(f)&=&\sum_{k\le n/2} \frac{n!}{2^k k! (n-2k)!} I_{n-2k}(\tr^k f)\nonumber\\
I_n(f)&=&\sum_{k\le n/2} \frac{(-1)^k n!}{2^k k! (n-2k)!} S_{n-2k}(\tr^k f)
\label{e.2.15}
\end{eqnarray}

For this result and other results see
\cite{bud},  \cite{humeyer3},  \cite{johnson}  and  the references therein.

\begin{example}  If $f_1\,, \cdots\,, f_n\in H$, then
\[
S_n( f_1\hatotimes  f_2\hatotimes \cdots \hatotimes  f_n\ )= \tilde f_1  \tilde f_2\cdots \tilde f_n\,.
\]
\end{example}

\setcounter{equation}{0}

\section{Malliavin calculus}
If $f_n\in H^{\hatotimes n}$ and $g\in H$,  then $\langle f_n\,, g\rangle_H$
is an element in  $H^{\hatotimes (n-1)}$. For  $f_n\in H^{\hatotimes n}$,
we define
\[
D_g I_n(f_n) =nI_{n-1}\left( \langle f_n\,, g\rangle_H\right)\,.
\]
If for almost every $x\in B$, the above right hand   is a continuous
functional of $g$ on $H$, then
\[
D I_n(f_n) =nI_{n-1}\left(   f_n \right)
\]
is a random variable with values in $H$.
We can extend $D_g$ and $D$ to general random variable $F$
by linearity and limiting argument.

It is easy to check that $D(FG)=FDG+GDF$.

In the same way we can introduce higher order derivatives $D^k F:
B\rightarrow H^{\hat\otimes k}$.

The space $D^{k, p}$ is defined as
\[
D^{k, p}=\left\{ F: B\rightarrow \RR\,;\ \
\|F\|_{k,p}^p:=\sum_{i=0}^k \EE \|D^i F\|_{H^{\hat\otimes
i}}^p<\infty\right\}\,.
\]
To describe this space, one may introduce the Ornstein-Uhlenbeck operator
$L$ defined by
\[
LF=\sum_{n=1}^\infty n F_n\,,\quad \hbox{if $F$ has the chaos expansion}\quad  F=\sum_{n=1}^\infty   F_n
\]
(which is the generator of the semigroup $P_tF=\Ga(e^{-t}) F$).

The following result is called the Meyer's inequality (see \cite{meyer3} and
also \cite{pisier} for a simpler analytic  proof).

\begin{theorem}    There is a constant $c_{k, p}$ and $C_{k,p}$ such that
\[
c_{k,p} \|(L+1)^{k/2} F\|_p\le \|F\|_{k,p}\le C_{k,p} \|(L+1)^{k/2} F\|_p\,.
\]
\end{theorem}
Meyer's inequality can be used to give a detailed description of $D^{k,p}$ even for non integer
$k$.

\begin{example} Let $f\in H$. Then
\[
\vare(f):=\sum_{n=0}^\infty \frac1{n!}I_n(f^{\otimes
n})=\exp\left(\tilde f-\frac12 \|f\|_H^2\right)\,
\]
is called an  {\it exponential vector} (in $L^2(B, \cB, \mu)$).

For an exponential vector $\vare(f)$ and a $g\in H$, we have
\[
D_g \vare(f)=\vare(f)\langle f, g\rangle\,,\quad D  \vare(f)=\vare(f)  f \,.
\]
\end{example}

 Let
\begin{eqnarray*}
\cS&=&\Bigg\{ F=h(\tilde f_1\,, \cdots \,, \tilde f_n)\,, \ \
f_1, \cdots, f_n\in H\quad {\rm and}\quad \hbox{ $h$ is smooth }\\
&&\qquad\quad \hbox{function on $\RR^n$\,, $n\ge 1$}\Bigg\}\,.
\end{eqnarray*}

\begin{example}
If $F=h(\tilde f_1\,, \cdots \,, \tilde f_n)$ is in $\cS$
with   $h$ being  of polynomial growth,
then $F\in D_{k,p}$ for all $k\ge 0$ and $p\ge 1$ and
\[
D_gF=\sum_{i=1}^n \frac{\partial h}{\partial x_i}
(\tilde f_1\,, \cdots \,, \tilde f_n) \langle f_i\,, g\rangle
\]
and
\[
D F=\sum_{i=1}^n \frac{\partial h}{\partial x_i} (\tilde f_1\,, \cdots \,, \tilde f_n)   f_i\,.
\]
\end{example}

\begin{definition}
If $F:B\rightarrow H$ and there is a random variable $Z$ such that
\begin{equation}
\EE\left( \langle F\,, DG\rangle\right)=\EE(ZG)\quad \forall \ G\in \cS\,,
\end{equation}
Then we say the divergence of $F$ exists and we denote it by $Z=\de(F)$.
\end{definition}

This means that the divergence operator $\de$ is the adjoint operator of the derivative operator
$D$.

\begin{example} If $g\in H$, then $\de(g)=\tilde g$.
\end{example}

\setcounter{equation}{0}

\section{   Wick product}

\begin{definition}
If $f_n\in H^{\hat\otimes n}$ and $g_m\in  H^{\hat\otimes m}$, then
the Wick product of $I_n(f_n)$ and $ I_m(g_m)$ is defined as
\[
I_n(f_n)\diamond I_m(g_m)=I_{n+m}(f_n\hat\otimes g_m)\,,
\]
where $f_n\hat\otimes g_m$ denotes the symmetric tensor product of
$f_n$ and $g_m$. If $F=\sum_{n=0}^{N_1}  I_n(f_n)$ and
$G=\sum_{m=0}^{N_2} I_m(g_m)$, then we define
\[
F\diamond G=\sum_{n=0}^{N_1}\sum_{m=0}^{N_2} I_{n+m}(f_n\hat\otimes
g_m)\,,
\]
\end{definition}
By a limiting  argument, we can extend the Wick product to general random variables
(see \cite{meyer} and \cite{kallianpur} for example).
\begin{remark}
 Of course, when we use ``by limiting argument" the definition depends on the
 topology that we   use. We can approximate $F$ and $G$ by finite combination
 of multiple integrals, $F_N$ and $G_N$, and define $F\diamond G$ as the limit of
 $F_N\diamond G_N$. Different choices of the topology (for example, in probability,
 $L^p$, almost surely etc) will lead to different
 definitions  of the Wick product.  In this paper, we shall use the $L^2$ limit.
 \end{remark}

It is clear that $F, G\in L^2$ does not imply that $F\diamond G$ is
a well-defined object in $L^2$. Now we present a sufficient
condition on $F$ and $G$ such that $F\diamond G\in L^2$. To this
end, we need to introduce some new norms.

If $F=\sum_{n=0}^\infty I_n(f_n)$, we define
\[
\|F\|_{(r)}^2 =\|\Ga(r) F\|^2=\sum_{n=0}^\infty n! r^{2n}
\|f_n\|_2^2\,.
\]

The following proposition can be found in \cite{yan95}(Theorem 3.1).
\begin{proposition}  Let $\displaystyle \frac1{p^2}+\frac1{q^2}= \frac1{r^2}$, $p, q,
r>0$. Then
\[
\|F\diamond G\|_{(r)} \le  \|F\|_{(p)}  \|G\|_{(q)}
\]
\end{proposition}
\begin{proof}
Let $F=\sum_{n=0}^\infty I_n(f_n)$ and  $G=\sum_{n=0}^\infty
I_n(g_n)$. Denote $h_n=\sum_{k+j=n} f_k\hat\otimes g_j $. Let
$a=r^{-2}q^2-1$. Then $1+a^{-1}=r^{-2}p^2$. We have
\begin{eqnarray*}
\sqrt{n!} \|h_n\|&\le& \sqrt{n!}\sum_{k+j=n}\|f_k\|\|g_j\|
 =  \sum_{k+j=n}\comb{n}{k}^{1/2}\sqrt{k!j!}\|f_k\|\|g_j\| \\
&\le&\left( \sum_{k+j=n} \comb{n}{k}a^k\right)^{1/2}\left( \sum_{k+j=n} a^{-k}k!j! \|f_k\|^2  \|g_j\|^2 \right)^{1/2} \\
&=& \left( \sum_{k+j=n}(1+a)^na^{-k}k!j! \|f_k\|^2  \|g_j\|^2\right)^{1/2}\\
&=& \left( \sum_{k+j=n}(1+a^{-1})^kk! \|f_k\|^2(1+a)^jj!
\|g_j\|^2\right)^{1/2},
\end{eqnarray*}
which implies the result.
\end{proof}

As a direct consequence of the above proposition we obtain a
condition on $F$ and $G$ such that $F\diamond G$ is in $L^2$.

\begin{theorem} If $\Ga(p)F\,, \Ga(q) G$ are in $ L^2$ with
$\frac1{p^2}+\frac1{q^2}= 1$, then $F\diamond G$ exists as an
element in $L^2$.
\end{theorem}

A useful tool in studying the Wick product is the so-called
$S$-transformation.
\begin{definition} Let $F\in L^p$ for some $p>1$. Then for any $\xi\in H$
$F(\cdot+\xi):\ B\rightarrow \RR$ is well-defined integrable random
variable. The following functional from $H$ to $\RR$
\begin{equation}
S(F)(\xi)=\EE \left[ F(\cdot+\xi)\right] \,,\quad \forall \ \xi\in
H\,.
\end{equation}
is called the {\it $S$-transformation} of $F$.
\end{definition}

By Cameron-Martin theorem we have
\begin{equation}
S(F)(\xi)=\EE [ F\vare(\xi)] \,,\quad \forall \ \xi\in H\,.
\end{equation}
Consequently, if $F=\sum_{n=0}^\infty  I_n(f_n)$, then
\begin{equation}
S(F)(\xi)=\sum_{n=0}^\infty n!\langle f_n, \xi^{\otimes n}\rangle
\,,\quad \forall \ \xi\in H\,.
\end{equation}
This implies that
\begin{equation}
S(F\diamond G)(\xi)=S(F)(\xi) S(G)(\xi)
\end{equation}
for suitable $F$ and $G$. For $f\in H$, since
$\vare(f)=\sum_{n=0}^\infty \frac1{n!}I_n(f^{\otimes n})$, we have
$$S\vare(f)(\xi)=\exp\{\langle f, \xi\rangle\},$$ and consequently
\begin{equation}
\vare(f)\diamond \vare(g)=\vare(f+g)\,, \quad \forall \ f,g\in
H\,.\label{e.4.5}
\end{equation}
We refer to \cite{huangyan} \cite{yan2},  \cite{kon}
and the references therein for more details in the framework of
white noise analysis.

\begin{proposition} \label{t.4.2} If $f_1\,, f_2\in H$
 are two unit  vectors which are  orthogonal, then
\[
H_n(\tilde f_1)\diamond H_m(\tilde f_2)= H_n(\tilde f_1) H_m(\tilde
f_2).\]
\end{proposition}

Under some suitable condition on $F$ and $G$, we have
\begin{eqnarray}
D_g(F\diamond G)
&=& D_gF\diamond G+  F\diamond D_gG\\
D (F\diamond G)
&=& D F\diamond G+  F\diamond D G
\end{eqnarray}

The following proposition can be  found in \cite{yan95}  (Theorem 5.5).

\begin{proposition} \label{t.4.8}
If  $g\in  H$,   $F\in  L^2(B, \cB, \mu)$ and  if
$ D_g F$ exists and is in  $  L^2(B, \cB, \mu)$, then $F\diamond \tilde g$ exists
in $L^1(B, \cB, \mu)$ and
\begin{equation}\label{e.4.4}
F\diamond \tilde g
=F \tilde g-  D_g F =F \tilde g-  \langle DF, g\rangle _H \,.
\end{equation}
\end{proposition}
\begin{proof}
Let
\[
\vE=\left\{ \sum_{finite} a_ie^{\tilde f_i}\,,\quad
\hbox{where}\quad a_i\in \RR\,, \ \ f_i\in H\right\}\,.
\]
Then $\vE$ is dense in $L^2$.  It is easy to see that \eref{e.4.4} is true for all
elements in $\vE$ by \eref{e.4.5}. A density argument shows the proposition.
\end{proof}

\begin{proposition}
If  $g\in  H$,   $F\in  L^2(B, \cB, \mu)$ and  if
$ D_g F$ exists and is in  $  L^2(B, \cB, \mu)$, then $F\diamond \tilde g$ exists
in $L^1(B, \cB, \mu)$ and
\begin{equation}
F\diamond \tilde g
=\de(Fg)\,.
\end{equation}
\end{proposition}
\begin{proof} Let $G\in \cS$.  Then
\begin{eqnarray*}
\EE\left( (F\diamond g)G\right)
&=& \EE\left(F G\tilde g-  G\langle DF, g\rangle _H  \right)\\
 &=& \EE\left(\langle D(F G)\,, g\rangle_H -   \langle GDF, g\rangle _H  \right)\\
&=& \EE\left(\langle    FDG\,,  g\rangle_H   \right)
 = \EE\left(\langle     DG\,,
Fg\rangle_H   \right)\,.
\end{eqnarray*}
Since $G$ is arbitrary, we show the proposition.
\end{proof}

The following proposition is used in \cite{duncan} and \cite{mams}.
\begin{proposition}
\begin{equation}
\EE\left[\left( F\diamond \tilde f\right)\left(G\diamond \tilde g\right)\right]
=\EE \left[ FG\langle f\,, g\rangle_H+D_gFD_fG\right]\,.
\end{equation}
\end{proposition}
\begin{proof}
\begin{eqnarray*}
\EE\left[\left( F\diamond \tilde f\right)\left(G\diamond \tilde g\right)\right]
&=& \EE\left[\de\left( F  f\right)\left(G\diamond \tilde g\right)\right]
=\EE\left[  F  D_f\left(G\diamond \tilde g\right)\right]\\
&=&\EE\left[  F   \left(D_f G\diamond \tilde g+ G\diamond D_f\tilde g\right)\right]\\
&=&\EE\left[  F   \left(D_f G\diamond \tilde g+ G\langle f, g\rangle_H\right)\right]\\
&=& \EE\left[  F   \de \left(D_f G  g\right)+F  G\langle f, g\rangle_H\right]\\
&=& \EE \left[ FG\langle f\,, g\rangle_H+D_gFD_fG\right]\,.
\end{eqnarray*}
This is the proposition.
\end{proof}

Let $F=\exp\{ s\tilde f-s^2 \|f\|^2/2\}$  and $G=\exp\{t \tilde
g-t^2 \|g\|^2/2\}$, where $s$ and $t$ are two arbitrary constants
(we use $\|\cdot\|$ to denote $\|\cdot\|_H$). We have
\begin{eqnarray*}
F\diamond   G
&=&\exp\{ s\tilde f+t\tilde g-\|sf+tg\|^2/2\}\\
&=&\exp\{ s\tilde f  - s^2\| f \|^2/2\} \exp\{  t\tilde g-t^2\|
g\|^2/2\}
\exp\{ - st\langle  f, g\rangle  \}  \\
&=&\sum_{i, j, p=0}^\infty s^i t^j (st)^p
 {(-1)^p I_i
(f^{\hat\otimes i})I_{j}(g^{\hat\otimes j})\langle  f, g\rangle
^p \over i!j!p!}\\
&=&\sum_{m, n=0}^\infty s^m t^n \sum_{p\le m\wedge n}
 {(-1)^p I_{m-p}
(f^{\hat\otimes m-p})I_{n-p}(g^{\hat\otimes n-p})\langle  f,
g\rangle ^p \over p!(n-p)!(m-p)!}.
\end{eqnarray*}
On other hand,  we have
$$
F\diamond   G=\sum_{n, m=1}^\infty s^mt^n I_m(f^{\hat\otimes m
})\diamond I_{n}(g^{\hat\otimes n})/m!n!\ .
$$
Comparing the coefficients of  $s^mt^n$,  we can write the above
formula as
$$
I_m(f^{\hat\otimes m })\diamond I_{n}(g^{\hat\otimes n})=\sum_{p\le
m\wedge n} {(-1)^p   n!m! \langle  f, g\rangle ^p  \over
p!(n-p)!(m-p)!}I_{m-p} (f^{\hat\otimes m-p})I_{n-p}(g^{\hat\otimes
n-p}).
$$
By using the Malliavin  derivative,  the above formula can be written as
$$
I_m(f^{\hat\otimes m })\diamond I_{n}(g^{\hat\otimes n})=\sum_{p\le
m\wedge n} {(-1)^p \over p!} \langle  D  ^p I_m (f^{\hat\otimes m
}),
  D  ^pI_{n }(g^{\hat\otimes n })\rangle_ {H^{\hat\otimes p}},
$$
where  $  D   ^pF$  is identified as a mapping from $\Om$  to
$H^{\hat\otimes p}$. By the polarization technique,  we have if $f_m$ and
$g_n$ are continuous symmetric functions of $m$ and $n$-variables,
then
\begin{equation}\label{e.4.6}
I_m(f_m )\diamond I_{n}(g_n)=\sum_{p\le m\wedge n} {(-1)^p \over p!}
\langle  D  ^p I_{m } (f_m),   D  ^pI_n (g_n)\rangle
_{H^{\hat\otimes p} }.
\end{equation}
In the same way we can obtain that
\begin{equation}\label{e.4.7}
I_m(f_m ) I_{n}(g_n)=\sum_{p\le m\wedge n} {1 \over p!} \langle  D
^p I_{m } (f_m)\,, \diamond    D  ^pI_n (g_n)\rangle
_{H^{\hat\otimes p} }\,,
\end{equation}
where   we consider $D  ^p I_{m } (f_m)$ and $ D  ^pI_n (g_n)$ as
two random variables with values in $H^{\hat\otimes p}$ and $\langle
D  ^p I_{m } (f_m)\,, \diamond    D  ^pI_n (g_n)\rangle
_{H^{\hat\otimes p} }$  is the {\it Wick scalar product}. More
precisely,
\begin{eqnarray*}
&&\langle  D  ^p I_{m } (f_m)\,, \diamond    D  ^pI_n (g_n)\rangle
_{H^{\hat\otimes p} }\\
&=&\sum_{k_1, \cdots, k_p=1}^\infty \langle  D  ^p I_{m } (f_m)\,,
e_{k_1}\hat\otimes \cdots \hat\otimes e_{k_p} \rangle
_{H^{\hat\otimes p} } \diamond \langle   D  ^pI_n (g_n)\,,
e_{k_1}\hat\otimes \cdots \hat\otimes e_{k_p} \rangle
_{H^{\hat\otimes p} }\\
&=&  \sum_{k_1, \cdots, k_p=1}^\infty m(m-1)\cdots (m-p+1) n(n-1)\cdots (n-p+1)
I_{m+n-2p}(h)\
\end{eqnarray*}
with
\begin{eqnarray*}h
&=&  \sum_{k_1, \cdots, k_p=1}^\infty \langle   f_n \,,
e_{k_1}\hat\otimes \cdots \hat\otimes e_{k_p} \rangle \hat\otimes
\langle g_n \,,  e_{k_1}\hat\otimes \cdots \hat\otimes e_{k_p}
\rangle\,,
\end{eqnarray*}
where as before,   $\{e_1, e_2\,, \cdots\}$ is an orthonormal basis of $H$.

We cab define $\langle  D  ^p G\,, \diamond    D  ^pG\rangle
_{H^{\hat\otimes p} } $ in a similar way.

The above two formulas are for single chaos.
We can use the linearity and limiting argument to show that
\begin{theorem} \label{t.5.3} If all  the  Malliavin derivatives of $F$ and $G$ exists
and satisfy for $H=F$ and $H=G$
\begin{equation}
\sum_{p=1}^\infty  {1 \over p!} \|  D  ^p H\|_{H^{\hat\otimes p}}^2
<\infty\,,  \label{e.5.9}
\end{equation}
then $F\diamond G$ exists as an element in $L^1(\Om, \cF, P)$ and
\begin{equation}
F\diamond   G=\sum_{p=0}^\infty {(-1)^p \over p!} \langle  D  ^p F,
  D   ^pG\rangle  _{H^{\hat\otimes p}}. \label{e.4.10}
\end{equation}
\end{theorem}
We can also try to find conditions for the following identity to hold.
\[
F  G=\sum_{p=0}^\infty {1 \over p!} \langle  D  ^p F\,, \ \diamond
  D  ^pG\rangle _{H^{\hat\otimes p} }\,.
  \]

\begin{example} If $G=\tilde g$ for some $g\in H$,  then from \eref{e.4.10}
\[
F\diamond G=F\tilde g-\langle DF, g\rangle\,.
\]
\end{example}

\begin{example} If $G=I_2(g_2)$ for some $g_2\in H^{\hatotimes 2}$,  then from \eref{e.4.10}
\[
F\diamond G=FI_2(g_2)-2\langle DF, I_1(g_2)\rangle+\langle D^2F, g_2\rangle\,.
\]
\end{example}

The following proposition states if $F$ and $G$ are independent and if they are in single chaos form,
then their Wick product and the usual product are the same.

\begin{proposition} If $F=I_n(f_n)$ and $G=I_m(g_m)$ are independent, where
$f_n\in H^{\hatotimes n}$ and $g_m\in H^{\hatotimes m}$, then
\begin{equation}
F\diamond G=FG\,.
\end{equation}
\end{proposition}
\begin{proof} It is a direct consequence of \eref{e.4.6} and Theorem 2.7.
\end{proof}

\medskip
From this proposition it is natural to conjecture   that if $F$ and $G$ are independent,
then $F\diamond G=FG$.

Denote by  $\tau_\xi F$ the translation operator:  $\tau_\xi
F(\om)=F(\om+\xi)$\,, \ $\om \in B$.   In the framework of white
noise analysis, the following identities are from \cite{yan95}.  It
holds in our framework here under suitable conditions which are not
made  precise here.
\begin{eqnarray}
\tau_\xi  F&=&\left[\vare(\xi) F\right]\diamond \vare(\xi)\\
\tau_\xi (F\diamond G) &=& (\tau_\xi  F)\diamond  (\tau_\xi
  G)\\
D(F\diamond G)&=& (DF)\diamond  G+ F\diamond (D G) \,.
\end{eqnarray}

\setcounter{equation}{0}
\section{Wick renormalization}

It is natural to define the Wick power
\[
F^{\diamond k}= \overbrace{F\diamond F\diamond \cdots \diamond F}^k\,.
\]
If $l\in H$,  then
\begin{eqnarray*}
\tilde l^{\diamond n}
&=&I_{n}\left(l^{\hat\otimes n}\right)=\|l\|^n H_n\left(\frac{\tilde l}{\|l\|}\right)\,;\\
\exp^{\diamond} \left(\tilde l\right)
&=& \exp\left(\tilde l-\frac{\|l\|^2}{2}\right)\,.
\end{eqnarray*}

If $f:\RR^d\rightarrow \RR$ is a real valued entire function, then
\begin{equation}
f(x)=\sum_{n=0}^\infty \sum_{n_1+\cdots+n_d=n} a_{n_1, \cdots, n_d} x^{n_1}\cdots x_d^{n_d}\,.
\end{equation}
\begin{theorem} \label{t.5.1}  Let
$X_1=\tilde l_1\,, \cdots\,, X_d=\tilde l_d$ be  independent Gaussian
random variables.  If $f:\RR^d\rightarrow \RR$ is real entire function given above such that
\begin{equation}
\sum_{n=0}^\infty \sum_{n_1+\cdots+n_d=n} n_1!\cdots n_d! a_{n_1, \cdots, n_d} \|X_1\|_2^{2n_1}
\dots \|X_d\|_2^{2n_d}<\infty\,, \label{e.5.2}
\end{equation}
then
\begin{equation}
f^{\diamond }(X_1, \cdots, X_d)=\sum_{n=0}^\infty \sum_{n_1+\cdots+n_d=n}
a_{n_1, \cdots, n_d} X_1^{\diamond
n_1} \diamond \cdots \diamond X_d^{\diamond n_d}\label{e.5.3}
\end{equation}
is well-defined as an element in $L^2$.  Moreover,
\begin{equation}
\EE\left[f^{\diamond }(X_1, \cdots, X_d)\right]^2 =\sum_{n=0}^\infty \sum_{n_1+\cdots+n_d=n} n_1!\cdots n_d! a_{n_1, \cdots, n_d} ^2\|X_1\|_2^2
\dots \|X_d\|_2^2\,.
\end{equation}
\end{theorem}
\begin{remark}
\begin{description}
\item{(i)}\quad It is obvious that $X_1^{\diamond
n_1} \diamond \cdots \diamond X_d^{\diamond n_d}$ in \eref{e.5.3} can be replaced by
$X_1^{\diamond
n_1}   \cdots   X_d^{\diamond n_d}$.
\item{(ii)}\quad We also denote
\[
:f (X_1, \cdots, X_d):=f^{\diamond }(X_1, \cdots, X_d)
\]
which is called the  {\it Wick ordering} or  {\it Wick
renormalization}  of $f (X_1, \cdots, X_d)$. A limiting argument can
be used  to discuss the case $d=\infty$.
\item{(iii)}\quad Wick renormalization is studied in \cite{yan91} in
the framework of white noise analysis and with the use of scaling
operator.
\end{description}
\end{remark}
\begin{proof} It is easy to see that
\[
\EE\left[ X_1^{\diamond
n_1} \diamond \cdots \diamond X_d^{\diamond n_d}\right]^2=n_1! \cdots,  n_d! \|X_1\|_2^{2n_1}
\dots \|X_d\|_2^{2n_d}\,.
\]
The theorem is  proved by the orthogonality of each term in \eref{e.5.3}.
\end{proof}

\begin{theorem}  Let $X=(X_1\,, \cdots\,, X_d)$ and $f$ be as in the previous theorem.
Let $(Y_1\,, \cdots\,, Y_d)$ be independent copy of   $X_1\,, \cdots\,, X_d$ ($Y_1\,, \cdots\,, Y_d$
and $X_1\,, \cdots\,, X_d$ are independent and have the same joint probability distribution).
Then
\begin{equation}
f^{\diamond }(X_1, \cdots, X_d)= \EE\left[ f(X_1+iY_1\,,
\cdots\,, X_d+iY_d)|X_1\,, \cdots\,, X_d\right]
\label{e.5.5}
\end{equation}
\end{theorem}
\begin{proof} We have (denote $\sum=\sum_{n=0}^\infty \sum_{n_1+\cdots+n_d=n}$)
\begin{eqnarray*}
\EE\left[ f(X+iY)|X\right]
&=&\sum
a_{n_1, \cdots, n_d} \EE \left[ (X_1+iY_1) ^{
n_1} \  \cdots   (X_d+iY_d)^{  n_d}\big|X_1\,, \cdots\,, X_d\right]\\
&=& \sum
a_{n_1, \cdots, n_d} \EE \left[ (X_1+iY_1) ^{
n_1}  \big|X_1 \right] \cdots \EE \left[ (X_d+iY_d) ^{
n_d}  \big|X_d \right] \,.
\end{eqnarray*}
Thus it suffices to show that for every $n\ge 1$,
\begin{equation}
\EE \left[ (X_1+iY_1) ^{
n }  \big|X_1 \right]=X_1^{\diamond  n }\,.
\label{e.5.6}
\end{equation}
In fact, we have for all $t\in \RR$,
\begin{eqnarray*}
\EE \left[ e^{t(X_1+iY_1)} \big|X_1 \right]
&=&  e^{t X_1-\frac12 \|Y_1\|_2^2} =e^{t X_1-\frac12 \|X_1\|_2^2} \\
&=&\sum_{n=0}^\infty \frac{t^n}{n!} X_1^{\diamond n}\,.
\end{eqnarray*}
Expanding  the left hand side in term of $t^n$ and comparing  the coefficients of $t^n$,   we
prove \eref{e.5.6} and hence  the theorem.
\end{proof}

Let us first use Theorem \ref{t.5.3} to compute
$:e^{\la X^2}:$,  where $X$ is a standard normal random variable.

\begin{proposition} If $|\la|<1$,  then the Wick renormalization of
$e^{\frac12 \la X^2}$ exists in the sense of Theorem \ref{t.5.1} and
\begin{equation}
:e^{\frac12 \la X^2}:
 =  \frac{1}{\sqrt{\la+1}}e^{\frac{\la}{2(\la+1)} X^2}\,.
\end{equation}
\end{proposition}
\begin{proof} First we have
\[
e^{\frac12 \la x^2}=\sum_{n=0}^\infty a_nx^n\,,
\]
where
\[
a_n=\cases{ 0& if $n=2k+1$\cr
\frac{\la^k}{2^k k!}&if $n=2k $\,. \cr
}\]
Thus \eref{e.5.2} is equivalent to $|\la|< 1$.

Let $Y$ be a standard normal independent of $X$.  From Theorem \ref{t.5.3}, we have
\begin{eqnarray*}
:e^{\frac12 \la X^2}:
&=&\EE\left[ e^{\frac12 \la \left(X+iY\right)^2}|X\right]\\
&=& \EE\left[ e^{\frac12 \la X^2 +i\la XY-\frac12\la Y^2 }|X\right]\\
&=& e^{\frac12 \la X^2 } \frac1{\sqrt{2\pi}}\int_{-\infty}^\infty
e^{ i\la Xy-\frac{\la+1}2  y^2 } dy\\
&=& \frac{1}{\sqrt{\la+1}}e^{\frac{\la}{2(\la+1)} X^2}\,.
\end{eqnarray*}
\end{proof}

\begin{example}  Let $f\in H^{\hatotimes 2}$ and  consider $:\exp\left(\frac12 I_2(f)\right):$.
\end{example}
It is well-known that there is an orthonormal basis $\{e_1, e_2, \cdots\}$ of $H$ and $\al_n\,, n=1, 2, \cdots$
such that $\sum_{n=1}^\infty \la_n^2<\infty$ and
\[
f=\sum_{n=1}^\infty \la_n e_n\hat\otimes e_n\,.
\]
Thus
\[
I_2(f)=\sum_{n=1}^\infty \la_n (\tilde e_n^2 -1)\,.
\]
Therefore
\begin{eqnarray*}
:\exp\left(\frac12 I_2(f)\right):
&=& \prod_{n=1}^\infty :e^{\frac12 \la_n (\tilde e_n^2-1)}:\\
&=& \prod_{n=1}^\infty e^{-\frac{\la_n}{2}-\frac12\log(\la_n+1)+ \frac{\la_n}{2(\la_n+1)} \tilde e_n^2}
\end{eqnarray*}

\begin{example} If we use the multiple Stratonovich integral, we have
\begin{eqnarray*}
:S_n(f_1\hatotimes f_2\hatotimes \cdots \hatotimes f_n):
&=& :\tilde f_1 \tilde f_2\cdots \tilde f_n:\\
&=& \tilde f_1 \diamond \tilde f_2\diamond \cdots \diamond\tilde f_n\\
&=& I_n(f_1\hatotimes f_2\hatotimes \cdots \hatotimes f_n)\,.
\end{eqnarray*}
\end{example}

Let $f\in H^{\hatotimes n}$ have all traces of order $k$, $k\le n/2$
and let $S_n^N(f)$ be given by \eref{e.2.sn}.   Then
\begin{eqnarray*}
:S_n^N(f):
&=&\sum_{k_1, \cdots, k_n=1}^N \langle f\,,
e_{k_1}\otimes \cdots \otimes e_{k_n}\rangle_{H^{\hatotimes n}}
\tilde
e_{k_1}\diamond \cdots \diamond \tilde e_{k_n}\\
&=&I_n \left( \sum_{k_1, \cdots, k_n=1}^N \langle f\,,
e_{k_1}\otimes \cdots \otimes e_{k_n}\rangle_{H^{\hatotimes n}}
e_{k_1}\otimes \cdots \otimes e_{k_n}\right)\\
&\rightarrow & I_n(f)\qquad (\hbox{in $L^2$})\,.
\end{eqnarray*}
Thus we see that if
existence of trace, we have if $f\in H^{\hatotimes n}$ have all traces of order $k$, $k\le n/2$,
then $S_n(f)$ exists as an element of $L^2$ and
\begin{equation}
:S_n(f):=I_n(f)\,.
\end{equation}

Combining the above with  the Hu-Meyer formula we have
\begin{proposition} If $f\in H^{\hatotimes n}$ have all traces of order $k$, $k\le n/2$,  then
\begin{eqnarray}
:I_n(f):=\sum_{k\le n/2} \frac{(-1)^k n!}{2^k k!(n-2k)!} I_{n-2k}\left({\rm Tr}^k f\right)\,.
\end{eqnarray}
\end{proposition}
\begin{proof} From the Hu-Meyer formula \eref{e.2.15}, we have
\begin{eqnarray*}
:I_n(f):&=&\sum_{k\le n/2} \frac{(-1)^k n!}{2^k k! (n-2k)!} :S_{n-2k}(\tr^k f):\\
&=&\sum_{k\le n/2} \frac{(-1)^k n!}{2^k k! (n-2k)!} I_{n-2k}(\tr^k f) \,.
\end{eqnarray*}
\end{proof}

It is also known in \cite{yan91} that
\begin{equation}
:F G :=(:F:) \diamond (:G:)\,.
\end{equation}

In Euclidean quantum field theory, the existence of interacting
field leads to the exponential integrability problem of
$S_n(f)$ for some very special $f$ on some specific abstract
Wiener space. The Wick renormalization method reduces the
problem to exponential integrability problem of
$I_n(f)$.  For more detailed  discussion see \cite{kallianpur}.

\end{document}